\documentclass{article}
\usepackage{amssymb,amsmath,amsthm,}
\usepackage{a4wide}
\usepackage{amscd}
\usepackage{amsfonts}
\usepackage{amssymb}
\usepackage{latexsym}
\usepackage{color}
\usepackage{esint}
\usepackage{MnSymbol}

\usepackage{mathrsfs}
\usepackage{comment}

\newtheorem{theorem}{Theorem}

\newtheorem{corollary}[theorem]{Corollary}
\newtheorem{definition}[theorem]{Definition}
\newtheorem{lemma}[theorem]{Lemma}
\newtheorem{proposition}[theorem]{Proposition}
\newtheorem{remark}[theorem]{Remark}

\let\e=\varepsilon

\let\d=\delta

\let\pt=\partial

\let\O=\Omega
\let\G=\Gamma

\let\g=\gamma
\let\th=\theta

\let\b=\beta

\newcommand\normmm[1]{\left\vert\kern-0.25ex \left\vert\kern-0.25ex \left\vert #1 \right\vert\kern-0.25ex \right\vert\kern-0.25ex \right\vert}

\newcommand\normm[1]{\left\lVert#1\right\rVert}
\newcommand\norm[1]{\left\lvert#1\right\rvert}

\newcommand\normmb[1]{\left\lbrack\kern-0.25ex\left\lbrack\ #1 \kern0.75ex\right\rbrack\kern-0.25ex\right\rbrack}

\newcommand\inn[1]{\left\langle#1\right\rangle}

\newcommand{\R}{\mathbb{R}}

\renewcommand{\P}{\mathbf{P}}
\renewcommand{\S}{\mathbf{S}}
\newcommand{\ip}{(\mathbf{I}-\mathbf{P})}
\newcommand{\ipc}{(\mathbf{I}-\mathbf{P}_\c)}
\newcommand{\be}{\begin{equation}}
\newcommand{\bee}{\begin{equation*}}

\newcommand{\bm}{\begin{multline}}
\newcommand{\ee}{\end{equation}}
\newcommand{\eee}{\end{equation*}}

\newcommand{\dd}{\mathrm{d}}

\renewcommand{\c}{\mathfrak{u} }

\newcommand{\1}{\mathbf{1}}

\title{Diffusive Expansion of the Boltzmann equation for the flow past an obstacle}
\author{Yan Guo, Junhwa Jung }
\date{}

\begin{document}

\maketitle

\makebox[\textwidth]{Dedicate to Professor Zhouping Xin's 65 birthday}

\begin{abstract}
The exterior domain problem is essential in fluid and kinetic equations. In this paper, we establish the validity of the diffusive expansion for the Boltzmann equations to the Navier-Stokes-Fourier system up to the critical time in an exterior domain with non-zero passing flow. We apply the $L^3-L^6$ framework to the unbounded domain in this paper. 
\end{abstract}

\section{Introduction}
The purpose of this paper is to establish the incompressible Navier-Stokes-Fourier expansion from the unsteady Boltzmann equation with diffusive boundary conditions in the exterior domain. The rescaled Boltzmann equation with small Knudsen and Mach numbers is given by the following equation
\begin{eqnarray}
&& \label{Boltzeqfull} \e \pt_t F + v \cdot \nabla_x F = \frac{1}{\e}Q(F,F), \quad \text{in } \R_{+} \times \Omega^c \times \R^{3}, \\
&& \label{Boltzbdd} F|_{\g_{-}} = \mathscr{P}^{w}_{\gamma}(F), \\
&& F(t,x,v)|_{t=0} = F_{0}(x,v) \ge 0, \quad \text{on } \Omega^{c} \times \R^3.
\end{eqnarray}
Where, $F(t,x,v)$ is the distribution for particles at time $t \ge 0$, position $x$ and velocity $v$. The Boltzmann collision operator $Q$ can be written as
\begin{eqnarray} \label{collisionoperator}
Q(f,g)(v) &=& Q^{+}(f,g) - Q^{-}(f,g), \\
Q^{+}(f,g)(v) &=& \iint_{\R^3 \times \S^2} B(\omega,{v-u}) f(v')g(u') \dd \omega \dd u,\\
Q^{-}(f,g)(v) &=& f(v) \iint_{\R^3 \times \S^2} B(\omega,{v-u}) g(u) \dd \omega \dd u,
\end{eqnarray}
where, $v'$ and $u'$ defined by
\begin{eqnarray*}
v' = v - \omega(v-u) \cdot \omega, \quad u' = u + \omega(v-u) \cdot \omega,
\end{eqnarray*}
and $B(\omega,{v-u}) $ is the cross section for hard potentials with Grad's angular cut off, so that $\int_{\{\norm{\omega}=1\}} B(\omega, V) \dd \omega \lesssim \norm{V}^{\th}$ for $0\le \th \le 1$, depending on the interaction potential.

Throughout this paper, $\Omega = \{x : \xi(x) <0 \}$ is a general bounded domain in $\R^3$. We denote its boundary as $\pt \Omega = \{x : \xi(x) = 0 \}$, where $\xi(x)$ is a $C^2$ function. We assume $\nabla \xi(x) \neq 0$ on $\pt \Omega$. The inner normal vector at $\pt \Omega$ is given by
\begin{eqnarray*}
n(x) = -\frac{\nabla \xi(x)}{\norm{\nabla \xi(x)}}
\end{eqnarray*}
and it can be extended smoothly near $\pt \Omega$. Let $\gamma = \pt \Omega \times \R^3 = \gamma_{+} \cup \gamma_{0} \cup \gamma_{-}$ with
\begin{eqnarray}
\label{defboundary}
\gamma_{\pm} = \left\{(x,v) \in \pt \Omega \times \R^3 : n(x) \cdot v \gtrless 0 \right\}, \quad \gamma_{0} = \left\{(x,v) \in \pt \Omega \times \R^3 : n(x) \cdot v = 0 \right\}
\end{eqnarray}

The interaction of the gas with the boundary $\pt \Omega$ is given by the diffusive boundary condition, defined as follows:
$$
M_{\rho, u, T} := \frac{\rho}{(2 \pi T)^{\frac{3}{2}}} exp\left[-\frac{\norm{v-u}^2}{2T}\right],
$$
be the local Maxwellian with density $\rho$, mean velocity $u$, and temperature $T$
and 
$$\mu = M_{1,0,1} = \frac{1}{(2 \pi )^{\frac{3}{2}}} exp\left[-\frac{\norm{v}^2}{2}\right]. $$
On the boundary, $F$ satisfies the diffusive reflection condition, defined as 
$$
F(t,x,v) = \mathscr{P}^{w}_{\gamma}(F)(t,x,v) \quad \text{on } \gamma_{-},
$$
where
$$
\mathscr{P}^{w}_{\gamma}(F)(t,x,v) := M^{w}(v) \int_{\R^3} F(t,x,w)(n(x) \cdot w) \dd w,
$$
with wall Maxwellian defined as
$$
M^{w} = \sqrt{2 \pi} \mu, \quad \int_{ \{v \cdot n \gtrless 0\} } M^{w}(v) \norm{n \cdot v} \dd v =1.
$$
Since the domain is the complement of a compact domain, we need to specify the conditions at infinity. 
Because we study the problem in the small Mach number
regime, we assume that the velocity at infinity is of order $\e$. In other words, we fix a constant vector $\mathbf{u}$ denoting,
\begin{eqnarray}
v_{\mathbf{u}} := v- \e \mathbf{u}, \quad \mu_{\mathbf{u}}(v) := \mu (v_{\mathbf{u}}) = M_{1,\e \mathbf{u}, 1},
\end{eqnarray}
we assume in a suitable sense
$$
\lim_{\norm{x} \to \infty} F(x,v) = \mu_{\mathbf{u}}(v).
$$

To derive the hydrodynamic limit let's think about Hilbert expansion near $\mu_{\mathbf{u}}$which gives as follows:
\begin{eqnarray}
\label{hilbertexpansion} F = \mu_{\mathbf{u}} + \sqrt{\mu_{\mathbf{u}}}\sum_{n=1}^{\infty} \e^{n} f_{n}.
\end{eqnarray}
To determine the coefficient of $f_1, \cdots f_n, \cdots $ in such expansion, we plug the formal expansion \eqref{hilbertexpansion} into \eqref{Boltzeqfull}:
\begin{eqnarray} \label{Boltzepower}
(\e \pt_t + v \cdot \nabla)\sum_{n=1}^{\infty} \e^n f_n = \frac{1}{\e \sqrt{\mu_{\mathbf{u}}}}Q(\mu_{\mathbf{u}} + \sqrt{\mu_{\mathbf{u}}}\sum_{n=1}^{\infty} \e^n f_n, \mu_{\mathbf{u}} + \sqrt{\mu_{\mathbf{u}}} \sum_{n=1}^{\infty} \e^n f_n).
\end{eqnarray}
To expand right hand side, we define $L_{\mathbf{u}}$, the linearlized Boltzmann operator as
\begin{eqnarray}
L_{\mathbf{u}} f := - \mu_{\mathbf{u}}^{-\frac{1}{2}} [Q(\mu_{\mathbf{u}}, \mu_{\mathbf{u}}^{\frac{1}{2}}f) + Q( \mu_{\mathbf{u}}^{\frac{1}{2}}f, \mu_{\mathbf{u}})] : = \nu_{\c} f - K f,
\end{eqnarray}
where $\nu_{\c} := \int_{\R^3 \times \{ \norm{\omega} = 1 \}}B(v_{\mathbf{u}}-v',\omega) \mu(v_{\c}) \dd v' \dd \omega$ is $0 \le \nu_0 \norm{v_{\mathbf{u}}}^{\th} \le \nu_{\c} \le \nu_1 \norm{v_{\mathbf{u}}}^{\th}$ and $K$ is a compact operator on $L^{2}(\R^3_{v})$. 
Define the nonlinear collision operator $\G$ as
\begin{eqnarray}
\G(f,g) = \frac{1}{\sqrt{\mu_{\mathbf{u}}}}Q(\sqrt{\mu_{\mathbf{u}}}f,\sqrt{\mu_{\mathbf{u}}}g).
\end{eqnarray}
Comparing the powers of $\e$ on both sides of the equation \eqref{Boltzepower}, we obtain the following equations:
\begin{eqnarray*}
0 &=& - L_{\mathbf{u}} f_1, \\
v_{\mathbf{u}} \cdot \nabla_x f_1 &=& - L_{\mathbf{u}} f_2 + \G(f_1,f_1),\\
\pt_t f_1 + v_{\mathbf{u}} \cdot \nabla_x f_2 + {\mathbf{u}} \cdot \nabla_x f_1  &=& -L_{\mathbf{u}} f_3 + \G(f_1,f_2) + \G(f_2,f_1), \\
&\vdots& \\
\pt_t f_{n} + v_{\mathbf{u}} \cdot \nabla_x f_{n+1} + {\mathbf{u}} \cdot \nabla_x f_{n}  &=& -L_{\mathbf{u}} f_{n+2} + \sum_{i+j=n+2, i,j\ge 1} \G(f_{i},f_{j}).
\end{eqnarray*}
From the above relation, we can obtain the fluid equations. We can easily check that  $L_{\mathbf{u}}$ is an operator on $L^{2}(\R^3_{v})$ whose kernel is 
\begin{eqnarray} \label{kerl}
\ker L_{\mathbf{u}} = \text{span} \{1, v_{\mathbf{u}}, \norm{v_{\mathbf{u}}}^2 \} \sqrt{\mu_{\c}}.
\end{eqnarray}
Let $\P_{\mathbf{u}}$ be the orthogonal projection on $\ker L_{\mathbf{u}}$.

Define $F$ as a small perturbation of $\mu_{\mathbf{u}}$, which satisfies the following relation:
\begin{eqnarray}
F = \mu_{\mathbf{u}} + \sqrt{\mu_{\mathbf{u}}}(\e f_1 + \e^2 f_2 + \e^{\frac{3}{2}}R),
\end{eqnarray}
where $R$ is the remainder.
The functions $f_1$ and $f_2$ are constructed by the solution of the fluid equation defined by the following:
\begin{eqnarray} \label{deff1}
f_1 = u \cdot v_{\mathbf{u}} \sqrt{\mu_{\mathbf{u}}},
\end{eqnarray}
and
\begin{eqnarray} \label{deff2}
f_2 &=& \frac{1}{2} \sum_{i,j=1}^{3} \mathscr{A}_{ij} \left[\pt_{x_i}u_{j} + \pt_{x_j}u_{i} \right]  + L_{\c}^{-1}\left[\G\left(f_1,f_1\right) \right],
\end{eqnarray}
where $\mathscr{A}_{ij}$ is given by
\begin{eqnarray}
\mathscr{A}_{ij} = -L_{\c}^{-1} \left(\sqrt{\mu} (v_{\mathbf{u}, i} v_{\mathbf{u},j} -\frac{\norm{v_{\mathbf{u}}}^2}{3} \delta_{ij} ) \right),
\end{eqnarray}
and  $u +\c$ is a solution of the Navier-Stokes equation
\begin{eqnarray}
\pt_t u + (u + \c) \cdot \nabla u + \nabla p = \kappa \Delta u,
\quad u = -\c \text{ on } \Omega^c, \quad u \to 0  \text{ as } \norm{x}  \to \infty.
\end{eqnarray}
According to Section \ref{nseqestimate}, the
Navier-Stokes equation has a unique short-time solution if the initial data $u_0 \in L^1 \cap L^\infty$. To show the boundedness of $R$, we need to justify the hydrodynamic limit to the Navier-Stokes equation. The equation for $R$ can be written as follows:
\begin{eqnarray} \label{BoltzR}
\e \pt_t R + v \cdot \nabla_x R + \e^{-1}L_{\c}(R) = h + \tilde{L}(R) + \e^{1/2} \G(R,R),
\end{eqnarray}
where $\tilde{L}(R) = \G(f_1+\e f_2,R) +\G(R,f_1+\e f_2)$ and
\begin{eqnarray*}
h&=&-\e^{1/2} \pt_t (f_1+\e f_2) -\e^{-1/2} v \cdot \nabla_x (f_1+\e f_2) - \e^{-3/2}L_\c(f_1+\e f_2) + \e^{-1/2} \G(f_1+\e f_2,f_1+\e f_2).
\end{eqnarray*}

\subsection{Notation Definition}

\begin{definition}
In this paper we define
\begin{eqnarray*}\normm{f(t)}_{L^{q}_{x}L^{r}_{v}} &:=& \left( {\int_{x}
\left({\int_{v} \norm{f(t)}^{r} \dd v}\right)^{q/r} \dd x ^{} }\right)^{p/q}, \\
\nu_{\c} &:=& \int_{\R^3}B(\omega, \norm{v_{\mathbf{u}}-u}) \mu(u_{\mathbf{u}}) \dd u, \\
\normm{f}_{\nu}&:=&\normm{\nu(v_{\mathbf{u}})^{1/2}f}_{2},\\
\norm{f}_{p,\pm}&:=& \left( \int_{\gamma_{\pm}} \norm{f}^p \dd \gamma \right)^{1/p} := \left( \int_{\gamma_{\pm}} \norm{f}^{p}  \norm{n(x) \cdot v}  \dd s \right)^{1/p}, \\
\inn{f,g} &:=& \int_{\R^3} fg \dd v.
\end{eqnarray*}
If $p$ or $q$ is infinity, then the integral should be replaced by the supremum norm.
\end{definition}

\subsection{Boundary condition}

In this paper, we will focus on the diffusive boundary condition, which is given in the following passage.

For $f \in L^1(\g_{\pm})$ we define
\begin{eqnarray}
\mathscr{P}_{\gamma} f = \mu^{-\frac{1}{2}} \mathscr{P}^{w}_{\gamma}[\mu^{\frac{1}{2}} f], 
\end{eqnarray}
where
$$
\mathscr{P}^{w}_{\gamma}(F)(t,x,v) := M^{w}(x,v) \int F(t,x,w)(n(x) \cdot w) \dd w,
$$
with wall Maxwellian defined as
$$
M^{w} = \sqrt{2 \pi} \mu, \quad \int_{ \{v \cdot n \gtrless 0\} } M^{w}(v) \norm{n \cdot v} \dd v =1.
$$
For $f \in L^{1}(\g_{\pm})$ we define
\begin{eqnarray}
&&\mathscr{P}^{\c}_{\g} f = \mu_{\c}^{-1/2} \mathscr{P}^{\omega}_{\g} [\mu_{\c}^{1/2} f] = \sqrt{2 \pi} \frac{\mu}{\sqrt{\mu_{\c}}} z_{\g_{+}} (f), \\
&&z_{\g_{\pm}} (f)(t,x) = \int_{\{v \cdot n (x) \gtrless 0 \}} \sqrt{\mu_{ \c}} \norm{v \cdot n(x)} f(t, x,v) \dd v.
\end{eqnarray}
The boundary condition for $R$ is
\begin{eqnarray}
R = \mathscr{P}^{\c}_{\g} R + \e^{1/2} r,
\end{eqnarray}
where 
\begin{eqnarray}
r = \mathscr{P}^{\c}_{\g}[f_2- \phi_{\e}] - (f_2- \phi_{\e}), \text{ on } \g_{-},
\end{eqnarray}
with $\phi_{\e}$ defined as 
\begin{eqnarray}
\phi_{\e} = \e^{-2} \mu_{\c}^{-1/2} [M_{1,\e (u + \c ),1} - \mu_{\c} - \e \sqrt{\mu_{\c}}f_1],
\end{eqnarray}
such that
\begin{eqnarray}
\norm{\phi_{\e}} \le C_{\b} (\norm{u}^2+  \norm{\c}^2)  \exp [- \b \norm{v}^2] \quad \text{ for any } \b<\frac{1}{4}.
\end{eqnarray}
From the definition of $r$ it follows that 
\begin{eqnarray} \label{eq27}
    \norm{r}_{2,+} +\norm{r}_{\infty} \lesssim \norm{\c}+ \norm{u_0}.
\end{eqnarray}

\subsection{Main Result}
The main result of this paper is the following theorem.
\begin{theorem} \label{mainthm}
Let $\Omega$ be a $C^\infty$ bounded open set in $\R^3$, and let $\Omega^c = \R^3 \backslash \bar{\Omega}$. For any $0< M $, consider the following boundary value problem:
\begin{equation} \label{thm1eq}
\left\{ \begin{array}{l l}
\e \pt_t F + v \cdot \nabla_x F = \frac{1}{\e}Q(F,F) & \quad \text{in } \R_{+} \times \Omega^c \times \R^{3}, \\
F|_{\g_{-}} = \mathscr{P}^{w}_{\gamma}(F) & \quad \text{on } \g_{-},\\
\lim_{\norm{x} \to \infty} F(t,x,v) = \mu_\c (v), & \\
F(t,x,v)\big|_{t=0} = F_{0}(x,v) \ge 0 & \quad \text{on } \Omega^{c} \times \R^3.
\end{array} \right.
\end{equation}
Suppose the initial datum takes the form $F_0 = \mu_\c + \e \sqrt{\mu_\c} f_0 \ge 0$ such that
\begin{equation}
f_0 =  f_{1,0} + \e f_{2,0}
\end{equation}
and
\begin{eqnarray}
f_{1,0} &=&  u_0(x) \cdot v_\c \sqrt{\mu_\c},\\
f_{2,0} &=& \frac{1}{2} \sum_{i,j=1}^{3} \mathscr{A}_{ij} \left[\pt_{x_i}u_{0,j} + \pt_{x_j}u_{0,i} \right] \\
&& + L^{-1}\left[\G\left(f_{1,0} ,f_{1,0} \right) \right],
\end{eqnarray}
satisfies
\begin{eqnarray}
\nabla \cdot u_0 &=& 0,\\
\norm{\c}, \normm{u_0}_{L^{p}}, \normm{\nabla u_0}_{L^p}, \normm{u_0}_{L^{p}}, \normm{\nabla u_0}_{L^p} &<& M, \\
\normm{R_0}_{L^2_{x,v}} + \e^{-1}\normm{\ipc R_0}_{L^2_{x,v}} + \normm{\pt_t R_0}_{L^2_{x,v}} + \e^{1/2} \normm{\omega R_0}_{L^\infty_{x,v}}  &<& M,
\end{eqnarray}
where $\omega(v) = e^{\b \norm{v}^2}$ with $0 < \b \ll 1$ and $p \in [1, \infty]$.
For any $T_0>0$ if the solution of the Navier-Stokes exist until time $T_0$ and 
\begin{align}
\normm{u(t)}_{L^{p}} \lesssim (\normm{u_0}_{L^1} +\normm{u_0}_{L^\infty}),
\end{align}
holds for $0\le t\le T_0$ and $p \in [1,\infty]$ .

Then there exist $\e(T_0, M) > 0$ such that
the problem \eqref{thm1eq} has a unique solution in $t \in [0,T_0]$ and it can be represented as
\begin{eqnarray}
F= \mu + \e \sqrt{\mu}(f_1 +\e f_2+ \e^{1/2}R),
\end{eqnarray}
with $f_1$ and $f_2$ given by \eqref{deff1} and \eqref{deff2}.
\end{theorem}

\begin{remark}
Note that we can set $T_0$ to be any large number. For small initial data, it is known that the Navier-Stokes equation has a global solution. In this case, we can guarantee that the limiting problem holds for any time. For large initial data, the global existence of the Navier-Stokes equation is a well-known open problem. However, we can still prove that the hydrodynamic limit holds at least for the lifespan of the fluid solution.
\end{remark}
\begin{remark}
$\pt_t R_0$ doesn't make sense in general. However, it can be obtain from the equation by $\pt_t R_0 = -v \cdot \nabla_{x} R_0 - \e^{-1}L_{\c}(R_0) + h + \tilde{L}(R_0) + \e^{1/2} \G(R_0,R_0)$.
\end{remark}
\begin{remark}
We remark that $\e^{1/2} R$ is of higher order in $L^p$ for $2 \le p \le 6$. On the other hand, $\e^{1/2} R$ can be large enough in $L^{\infty}$ sense, which makes it possible that $(F-(\mu+\e \sqrt{\mu}f_{1})) / \e$ can be arbitrarily large in $L^{\infty}$.
\end{remark}

\subsection{Historical background}
In Hilbert's sixth question, he asked for the mathematical justification of the relationship between the Boltzmann equation and fluid equations in the context of small Knudsen numbers(\cite{Hilbert1900}, \cite{Hilbert1916}, \cite{Hilbert_}).
The present work partially answer the hydrodynamic limit of Boltzmann equation. In 1971, Sone discussed the hydrodynamic limit of Boltzmann equation in the steady case(\cite{Sone1971}). Caflisch and Lachowicz obtained the Hilbert expansion of the Boltzmann equation, which relates it to the solution of the compressible Euler equations for small Knudsen numbers(\cite{Caflisch1980}, \cite{Lachowicz1987}). On the other hand, Nishida, Asano, and Ukai proved it using different methods(\cite{Nishida1978},  \cite{Asano1983}).

We can obtain the incompressible Navier-Stokes equation from the Boltzmann equation when both the Mach number and Knudsen number go to zero. In the small Mach number regime, it is natural to consider perturbations near the absolute Maxwellian.

This fact was observed by Sone during the formal expansion for the steady solution (\cite{Sone1971}, \cite{Sone1987}). Among other contributions, Masi, Esposito, and Lebowitz provided results on short-time existence in $\R^d$(\cite{Masi1989}). Bardos and Ukai proved the incompressible Navier-Stokes limit in a distributional sense using semigroup theory (\cite{Bardos1991}). Golse and Saint-Raymond derived the convergence to the Diperna-Lions renormalized solutions(\cite{Masmoudi2003}, \cite{Lions2001}). Recently Guo introduced $L^2-L^\infty$ framework to prove incompressible Navier-Stokes limit to periodic domain and bounded domain(\cite{Guo2006}, \cite{Esposito2018-apde}).

On the other hand, in the context of the exterior boundary value problem for the Boltzmann equation, Ukai and Asano studied this problem with a fixed Knudsen number(\cite{Ukai1983}, \cite{Ukai1986}, \cite{Ukai2009}), Xia also demonstrated a similar result with relaxed assumptions using the spectral method(\cite{Xia2017}). For hydrodynamic limit as the Knudsen number goes to 0, Esposito, Guo and Marra proved the existence of a steady solution in an exterior domain(\cite{Esposito2018-cmp}). 

Our main result is the justification of the hydrodynamic limit of equation \eqref{thm1eq} for a fixed time with arbitrary initial data. The main point of the work is the presence of a nonzero speed at infinity. In this paper, we apply Guo's $L^2-L^\infty$ framework(\cite{Esposito2013}) and the $L^3-L^6$ framework(\cite{Esposito2018-apde}) to an unbounded domain. By doing this, we can obtain the hydrodynamic limit up to the lifespan of the Navier-Stokes equation.

\section{Linear Estimate}

\subsection{Energy estimate}
In this section, we will prove the existence and the energy estimate for the following linear problem:
\begin{equation} \label{lineareq}
\left\{ 
\begin{array}{l l}
\e \pt_t f + v \cdot \nabla f + \e^{-1}L_{\c} f =g &\quad (x,v) \in \Omega^c, \\
f = \mathscr{P}^{\c}_{\g}f + \e^{1/2} r & \quad (x,v) \in \g_{-}, \\
\lim_{\norm{x} \to \infty} f = 0.
\end{array}
\right.
\end{equation}
The following lemmas came from the \cite{Esposito2018-cmp}.

\begin{lemma}
{\label{green}} Assume that $f(x,v), \ h(x,v)\in
L^p(\Omega^c \times \R^3)$, $p \ge 2$ and $v \cdot \nabla_x f, v \cdot \nabla_x h  \in L^{\frac{p}{p-1}}(\Omega^c \times \R^3)$ and $f
\big|_{\g}, h\big|_{\g}\in L^2(\pt \Omega \times \R^3)$.
Then
\begin{eqnarray}
\iint_{\Omega^c \times \R^3} [(v \cdot \nabla_x h) f + (v \cdot \nabla_x f) h] \dd x \dd v = \int_{\g_+} f h \dd \g - \int_{\g_-} f h \dd \g.  \label{steadyGreen}
\end{eqnarray}
\end{lemma}
\begin{proof}
This is Lemma 2.1 of \cite{Esposito2018-cmp}.
\end{proof}

\begin{lemma}
\label{trace_s}
Assume $\Omega_1$ is an open bounded subset of $\R^{3}$
with $\pt (\Omega_1 \backslash \overline{\Omega})$ in $C^{2}$, such that $\{x\in \O^c\, |\, d(x,\O)\le 1\}\subset \O_1$. We define 
\begin{equation}  \label{non_grazing}
\gamma_{\pm}^{\delta} : = \{ (x,v) \in \gamma_{\pm} : \norm{n(x)\cdot v}  >
\delta, \ \ \delta\leq |v| \leq \frac{1}{\delta} \}.
\end{equation}
Then 
\begin{equation*}
\norm{f\mathbf{1}_{\gamma_{\pm}^{\delta}}}_{1,\pm} \lesssim_{\delta, \Omega_1} \normm{f}_{L^1(\Omega_1\backslash\Omega)} + \normm{ \pt_t f+ v\cdot \nabla_{x} f}_{L^1(\Omega_1 \backslash \Omega)} .
\end{equation*}
\end{lemma}
\begin{proof}
This is Lemma 2.2 of \cite{Esposito2018-cmp}.
\end{proof}
\begin{remark}
According to \cite{Esposito2013}, $\norm{\mathscr{P}^{\c}_{\g} f}_{2,\pm} \lesssim_{\d} \norm{\mathscr{P}^{\c}_{\g} f \1_{\g_{\pm}^\delta}}_{2,\pm}$ from previous lemma applied to $f^2$ we get
\begin{eqnarray*}
\norm{\mathscr{P}^{\c}_{\g} f}_{2,\pm}^2 
&\lesssim_{\delta} &\normm{f}_{L^2(\Omega_1\backslash\Omega)}^2 + \normm{ \pt_t f^2+ v\cdot \nabla_{x} f^2}_{L^1(\Omega_1 \backslash \Omega)} \\
&\lesssim & \normm{f}_{L^2(\Omega_1\backslash\Omega)}^2 + \normm{ fg}_{L^1(\Omega_1 \backslash \Omega)} + \normm{\e^{-1} f L_{\c} f}_{L^1(\Omega_1 \backslash \Omega)}.
\end{eqnarray*}
\end{remark}

\begin{lemma} \label{boundary1}
\begin{eqnarray} \label{ineq1}
\norm{\int_{\pt \Omega} \int_{\{v\cdot n>0\}} v\cdot n \norm{\mathscr{P}^{\c}_{\g} f}^2 \dd v \dd S -\int_{\pt \Omega} \int_{\{v\cdot n<0\}} \norm{v\cdot n} \norm{\mathscr{P}^{\c}_{\g} f}^2 \dd v \dd S} \lesssim \e\norm{\c}\int_{\g_+}\norm{f}^2\dd \g.
\end{eqnarray}
\begin{eqnarray}
\norm{\int_{\pt\Omega} \int_{\{v \cdot n > 0 \}} v \cdot n \mathscr{P}^{\c}_{\g} f (1-\mathscr{P}^{\c}_{\g}) f \dd v \dd S } \lesssim  \e \norm{\c} \int_{\g_+} \norm{f}^2\dd \g.\label{ineq2}
\end{eqnarray}
\end{lemma}
\begin{proof}
This is Lemma 2.4 of \cite{Esposito2018-cmp}.
\end{proof}

\begin{lemma}\label{boundary2} 
For any $\eta>0$,  
\begin{eqnarray} \label{ineq3}
\norm{\int_{\pt\O} \int_{ \{ v \cdot n < 0 \}}  \norm{v \cdot n} \e^{1/2} r  \mathscr{P}^{\c}_{\g} f \dd v \dd S} \lesssim \frac 1{\eta} \normm{z_\g(r)}_2^2+\e\eta \norm{f}_{2,+}^2+ \e^{3/2} \norm{r}_{2,-}^2
.
\end{eqnarray}
\end{lemma}
\begin{proof}
This is Lemma 2.5 of \cite{Esposito2018-cmp}.
\end{proof}

For fixed $\e$, the construction of the solution to the linear problem \eqref{lineareq} is standard (see e.g., \cite{Maslova1993}). The following proposition provides the energy estimate.
\begin{proposition}\label{energy} 
The solution to \eqref{lineareq}  
 satisfies the inequality
\begin{eqnarray}
\pt_t \normm{f}_2^2+\e^{-2} \normm{\ipc f}_\nu^2 +\e^{-1}\norm{(1-\mathscr{P}^{\c}_{\g})f}^2_{2,+} \lesssim  \normm{\nu^{-\frac 1 2}\ipc g}_2^2 + (1+\e^{\frac 1 2})\norm{r}_{2,-}^2 \\
+(\e\norm{\c})^{-1}\normm{z_\g(r)}_2^2 + \e^{-1} \normm{\P_{\c}g \P_{\c}f}_{1} + \norm{\c} \normm{\P_{\c} f}^2_{L^2(\O_1 \backslash \O)}.
\end{eqnarray}
\end{proposition}

\begin{proof}
Multiply $f$ to the equation \eqref{lineareq} and integrate by $x$ and $v$ variable. Then we can get
\begin{eqnarray}
\e \frac{1}{2} \pt_t \normm{f}_{L^2}^2 + \frac{1}{2}\iint_{\Omega^c \times \R^3} v \cdot \nabla f^2 \dd x \dd v + \e^{-1} \iint_{\Omega^c \times \R^3} L_{\c}(f)f \dd x \dd v
= \iint_{\Omega^c \times \R^3} fg  \dd x \dd v.
\end{eqnarray}
Then, by lemma \ref{green}, we can get
\begin{eqnarray}
\frac{1}{2}\pt_t \normm{f}_2^2 +\e^{-1} \frac{1}{2} \left(\int_{\g_+} f^2 \dd \g - \int_{\g_-} f^2 \dd \g \right)+ \e^{-2} \iint_{\Omega^c \times \R^3} L_{\c}(f)f \dd x \dd v = \e^{-1} \iint_{\Omega^c \times \R^3} fg  \dd x \dd v.
\end{eqnarray}
From the semi-positivity of $L_{\c}$
\begin{eqnarray*}
\iint_{\Omega^c \times \R^3} L_{\c}(f)f \dd x \dd v \gtrsim \normm{\ipc f}_2^2.
\end{eqnarray*}
By the Holder inequality,
\begin{eqnarray}
\e^{-1} \iint_{\Omega^c \times \R^3} fg  \dd x \dd v \le \e^{-1} \normm{\ipc f}_{\nu} \normm{\nu^{-1/2} \ipc g}_{2} + \e^{-1} \normm{\P_{\c}g \P_{\c}f}_{1}.
\end{eqnarray}
From the boundary condition on $\g_{-}$ we have $f = \mathscr{P}^{\c}_{\g} f + \e^{1/2} r$. By lemma \ref{boundary2}
\begin{eqnarray}
\int_{\g_-} f^2 \dd \g = \int_{\g_-} [\mathscr{P}^{\c}_{\g} f + \e^{1/2} r]^2 \dd \g = \int_{\g_-} (\norm{\mathscr{P}^{\c}_{\g} f}^2+\e \norm{r}^2+2\e^{1/2} r \mathscr{P}^{\c}_{\g} f ) \dd \g\\ 
\lesssim \int_{\g_-} \norm{\mathscr{P}^{\c}_{\g} f}^2 \dd \g + \e\norm{r}_{2,-}^2+\e^{3/2}\norm{r}_{2,-}^2+\frac 1{\eta}\normm{z_\g(r)}_2^2+\e \eta \norm{f}_{2,+}^2 .
\end{eqnarray}
Moreover,
\begin{eqnarray}
\int_{\g_+} f^2 \dd \g = \int_{\g_+} \norm{(1-\mathscr{P}^{\c}_{\g})f }^2 \dd \g + \int_{\g_+} \norm{\mathscr{P}^{\c}_{\g} f}^2 \dd \g + 2 \int_{\g_+} \mathscr{P}^{\c}_{\g} f (1-\mathscr{P}^{\c}_{\g})f  \dd \g.
\end{eqnarray}
The last term can be controlled by lemma \ref{boundary1}. Then collect all the term and put $\eta = \norm{\c}$ then we can get
\begin{eqnarray*}
\pt_t \normm{f}_2^2+\e^{-2} \normm{\ipc f}_\nu^2 +\e^{-1}\norm{(1-\mathscr{P}^{\c}_{\g})f}^2_{2,+} &\lesssim&  \normm{\nu^{-\frac 1 2}\ipc g}_2^2 + (1+\e^{\frac 1 2})\norm{r}_{2,-}^2 \\
&&+(\e\norm{\c})^{-1}\normm{z_\g(r)}_2^2 + \e^{-1} \normm{\P_{\c}g \P_{\c}f}_{1} + \norm{\c} \norm{P^\c_{\g} f}^2_{2,+}.
\end{eqnarray*}
Then last term of the above term can be controlled by
\begin{eqnarray}
\norm{\mathscr{P}^{\c}_{\g} f}_{2,+}^2\le \normm{f}_{L^2(\O_1\backslash\O)}^2+ \e^{-1}\normm{
\ipc f}_{L^2(\O_1\backslash\O)}^2+\normm{\nu^{-\frac 1 2}
g}_{L^2(\O_1\backslash\O)}^2
\end{eqnarray}
Then we split $f$ and $g$ into $\P_{\c}$ and $\ipc$ then we can get the desired result.
\end{proof}

\subsection{$L^p$ estimate}
In this section we will introduce $L^\infty$ estimate and $L^6$ estimate.

\begin{proposition}\label{linfestimate}
($L^\infty$ estimate)

If
$$
[\e \pt_t + v \cdot \nabla_x + \e^{-1}C_0 \inn{v} ]\norm{f} \le \e^{-1} K_{\b}\norm{f} + \norm{g}
$$
$$
\norm{f|_{\gamma_{-}}} \le \mathscr{P}^{\c}_{\g} \norm{f} +\norm{r}, \norm{f|_{t=0}} \le \norm{f_0}
$$
Then, for $\omega (v) = e^{\b' \norm{v}^2}$ with $0 < \b' << \b$
\begin{eqnarray*}
\normm{\e^{1/2} \omega f(t)}_{\infty} &\lesssim& \normm{\e^{1/2} \omega f_0}_{\infty} + \sup_{0\le s\le t } \normm{\e^{1/2} \omega r(s)}_{\infty} +\e^{3/2} \sup_{0\le s\le t } \normm{ \inn{v}^{-1} \omega g(s)}_{\infty} \\
&&+\sup_{0\le s\le t } \normm{ \P_\c f(s)}_{L^6(\Omega^{c})} + \e^{-1}\sup_{0\le s\le t } \normm{ \ipc f(s)}_{L^{2}(\Omega^c \times \R^3)} \\
\normm{\e^{1/2} \omega f(t)}_{\infty} &\lesssim& \normm{\e^{1/2} \omega f_0}_{\infty} + \sup_{0\le s\le t } \normm{\e^{1/2} \omega r(s)}_{\infty} +\e^{3/2} \sup_{0\le s\le t } \normm{ \inn{v}^{-1} \omega g(s)}_{\infty} \\
&&+ \e^{-1}\sup_{0\le s\le t } \normm{ f(s)}_{L^{2}(\Omega^c \times \R^3)}.
\end{eqnarray*}
\end{proposition}

\begin{proof}
This is a proposition 3.10 of the \cite{Esposito2018-apde}.
\end{proof}

\begin{proposition}($L^6$ estimate)\label{l6estimate}
We have:
\begin{eqnarray*}
\normm{\P_\c f}_{6}\lesssim \e^{-1}\normm{\ipc f}_\nu+
\normm{ g{{\nu}^{-\frac 1 2}}}_2 + \e \normm{\pt_t f}_{2}+\e^{-1/2}\norm{(1-\mathscr{P}^{\c}_{\g})f}_{2,+}\\+ \e^{1/2} \norm{r}_\infty+ o(1)[\e^{1/2}\normm{f}_{\infty}] 
.\label{P699}
\end{eqnarray*}
\end{proposition}

\begin{proof}
This is a proposition 3.2 of the \cite{Esposito2018-cmp}. In that paper they assume the smallness of $\c$ but we can get the same result without assuming that.
\end{proof}

\begin{proposition} \label{l3estimate}
($L^3$ estimate)
Assume $g \in L^2 (\R_+ \times \Omega^c \times \R^3)$, $f_0 \in L^2 (\Omega^c \times \R^3)$, and $f_\g \in L^2(\R_+ \times \g)$. Let $f \in L^{\infty}(\R_+ ; L^2(\Omega^c \times \R^3))$ solve (BE) in the sense of distribution and satisfy $f(t,x,v) = f_{\g} (t,x,v)$ on $\R_+ \times \g$ and $f(0,x,v) = f_0 (x,v)$ on $\Omega \times \R^3$. Then there exist $S_1$ and $S_2$ satisfying
\begin{align}
\norm{a(t,x)} + \norm{b(t,x)} + \norm{c(t,x)} \le S_1 f(t,x) + S_2 f(t,x).
\end{align}
Moreover,
\begin{align}
\normm{S_1 f}_{L^2_t L^3_x} &\lesssim \normm{f}_{L^2_t L^2_x}+\normm{g}_{L^2_t L^2_x} + \normm{f_0 }_{L^2_x} + \normm{ v \cdot \nabla f_0}_{L^2_x} + \norm{f_0}_{L^2_{\g}} + \norm{f}_{L^2_t L^2_{\g}} \\
\normm{S_2 f}_{L^2_t L^2_x} &\lesssim \normm{\ipc f}_{L^2_{t,x,v}}
\end{align}
\end{proposition}
\begin{proof}
This is a proposition 3.4 of the \cite{Esposito2018-apde}. In the paper \cite{Esposito2018-apde}, the proposition deals with the bounded domain case. However, the way they did is to extend the function to the $\R^3$. When we extend the equation, only what is important near the boundary matters, which means that we can apply the same method to the exterior domain.
\end{proof}

\section{Nonlinear estimate}
In this section we will estimate nonlinear term of the equation which is represent as $g$ in \eqref{lineareq}.
\begin{align*}
g = h + \tilde{L}(f) + \e^{1/2} \G(f,f)
\end{align*}

\subsection{Estimate of Navier-Stokes equation} \label{nseqestimate}
In this section we will discuss the local existence of Navier-Stokes equation in exterior domain.
Start with the incompressible viscous fluid past an isolated rigid body with the following initial boundary value problem of the Navier-Stokes equation:
\begin{align} \label{nsequ}
\begin{split}
\pt_t u - \Delta u + (u \cdot \nabla)u + \nabla P = 0, \quad \nabla \cdot u = 0 &\text{ in } (0,\infty) \times \Omega \\
u(t,x) = 0 & \text{ on } (0,\infty) \times \pt \Omega \\
u(0,x) = u_0(x) & \text{ in } \Omega \\
\lim_{x \to \infty} u(t,x) = \c.
\end{split}
\end{align}
Setting $v = u - \c$, the Navier-Stokes equation became
\begin{align} \label{nseqv}
\begin{split}
\pt_t v - \Delta v + (\c \cdot \nabla)v + (v \cdot \nabla)v + \nabla P = 0, \quad \nabla \cdot v = 0 &\text{ in } (0,\infty) \times \Omega \\
v(t,x) = -\c & \text{ on } (0,\infty) \times \pt \Omega \\
v(0,x) = u_0(x) -\c & \text{ in } \Omega \\
\lim_{x \to \infty} v(t,x) = 0.
\end{split}
\end{align}
Let $w$ is a stationary solution of \eqref{nseqv}. Which means
\begin{align} \label{nseqw}
\begin{split}
- \Delta w + (\c \cdot \nabla)w + (w \cdot \nabla)w + \nabla P = 0, \quad \nabla \cdot w = 0 &\text{ in }  \Omega \\
w(x) = -\c & \text{ on }  \pt \Omega \\
\lim_{x \to \infty} w(x) = 0.
\end{split}
\end{align}
Existence of $w$ and $w \in L^3(\Omega^c) \cap L^{\infty}$, $\nabla w \in L^2(\Omega^c) \cap L^\infty$ is proved in \cite{Galdi2011}.

Let $z = v-w$ then the equation become
\begin{align}
\label{nseqz}
\begin{split}
\pt_t z - \Delta z + (\c \cdot \nabla)z +(w \cdot \nabla)z+ (z \cdot \nabla)w  + (z \cdot \nabla)z + \nabla P = 0, \quad \nabla \cdot z = 0 &\text{ in } (0,\infty) \times \Omega \\
z(t,x) = 0 & \text{ on } (0,\infty) \times \pt \Omega \\
z(0,x) = z_0(x) := u_0(x) -\c -w  & \text{ in } \Omega \\
\lim_{x \to \infty} z(t,x) = 0.
\end{split}
\end{align}

Thus, the first step to understand above equation is to understand Oseen's equation given follows
\begin{align}
\begin{split} \label{oseeneq}
\pt_t u - \Delta u + (\c \cdot \nabla)u + \nabla P = 0, \quad \nabla \cdot u = 0 &\text{ in } (0,\infty) \times \Omega \\
u(t,x) = 0 & \text{ on } (0,\infty) \times \pt \Omega \\
u(0,x) = u_0(x) & \text{ in } \Omega \\
\lim_{x \to \infty} u(t,x) = 0.
\end{split}
\end{align}
To understand Ossen's equation, we define Oseen's semigroup and Leray projection.
\begin{definition}
Let $T_{\c}$ is a Ossen's semigruop, which means $T_{\c}(t)u_0$ is a solution of \eqref{oseeneq}.

Let $\mathbb{P}$ be the Leray projection which from $L^p(\Omega^c)^3$ to the complement of $\{u \in C^{\infty}_0(\Omega^c) : \nabla \cdot u =0 \text{ in } \Omega^c\}$.
\end{definition}

\begin{proposition} \label{lpdecay}
(Existence and time decay of the Oseen's equation.)
Let $T_{u_{\infty}}(t)$ is a analytic semigroup generated by Oseen's equation. Then, Oseen's semigruop is commute with Leray projection and the following decay estimate holds if $1 \le p \le q \le \infty$ and $(p,q) \neq (1,1), (\infty, \infty)$
\begin{align*}
\normm{T_{u_{\infty}}(t) u_0}_q &\le C (1+t)^{3(1/p-1/q)} \normm{u_0}_{p}, \\
\normm{\nabla T_{u_{\infty}}(t) u_0}_q &\le C (1+t)^{3(1/p-1/q)-1/2} \normm{u_0}_{p}.
\end{align*}
\end{proposition}
\begin{proof}
This is the result of the \cite{Enomoto2005} and \cite{Kobayashi1998}.
\end{proof}

Next we will prove the short time existence of the $L^p$ solution of the following nonlinear problem.

\begin{proposition}
(Local existence of Navier-Stokes equation) For any $M>0$, there exist $T_0$ such that the equation \eqref{nseqz} has a solution in $[0,T_0] \times \Omega^c$ if $\norm{\c}, 
\normm{z_0}_{L^1},\normm{z_0}_{L^\infty} < M$ and $\nabla \cdot z_0 = 0$
\end{proposition}
\begin{proof}
Define the Duhamel operator $\Phi$ as follows
\begin{align*}
\Phi(z) :&= T_{\c}(t) z_0 + \int_{0}^{t} T_{\c}(t-s) [(w \cdot \nabla)z+ (z \cdot \nabla)w  + (z \cdot \nabla)z] \dd s \\
& = T_{\c}(t) z_0 + \int_{0}^{t} T_{\c}(t-s) \mathbb{P}(\nabla\cdot (w \otimes z + z \otimes w + z \otimes z)) \dd s.
\end{align*}
Then, the solution of the \eqref{nseqz} is the fixed point of $\Phi$ which is $\Phi(z) = z$.
For $1< p< \infty$ we can obtain the following $L^p$ control, if we set $p=q$ in Proposition \ref{lpdecay} and use Holder inequality
\begin{align} \label{lpphi}
\normm{\Phi (z(\cdot,t))}_{L^p} \le C_1 \left( \normm{z_0}_{L^p} + t(\normm{w}_{L^\infty}+ \normm{z}_{L^\infty}) \normm{z}_{L^p}\right),
\end{align}
holds for constant $C_1$.
For $L^\infty$ control we set $p=p$ and $q = \infty$ in Proposition \ref{lpdecay},
\begin{align} \label{linfphi}
\normm{\Phi (z(\cdot,t))}_{L^\infty} \le C_2 \left(  \normm{z_0}_{L^p} + t(\normm{w}_{L^\infty}+ \normm{z}_{L^\infty}) \normm{z}_{L^p} \right),
\end{align}
holds for constant $C_2$.
Since $\norm{\c} < M$, $\normm{w}_{L^\infty} \le  C_3 M$ holds for consatant $C_3$ by \cite{Galdi2011}. 

Define space $X$ 
\begin{align*}
X := \left\{ z : \normm{z}_{L^1}, \normm{z}_{L^\infty} < C M  \right\},
\end{align*}
where $C = 2 \max\{C_1, C_2, C_3\}$ and define time $T_0 = \frac{1}{10 C M}$.

By \eqref{lpphi} and\eqref{linfphi}, if $z \in X$, then $\Phi (z) \in X$. In addition, $\Phi$ is a contraction on that space $X$ since
\begin{align*}
\normm{\Phi(u)- \Phi(v)}_{L^p} &\le T_0 (\normm{w}_{L^\infty}+ \normm{u}_{L^\infty} + \normm{v}_{L^\infty}) \normm{u-v}_{L^p} \\
\normm{\Phi(u)- \Phi(v)}_{L^\infty} &\le T_0 (\normm{w}_{L^\infty}+ \normm{u}_{L^\infty} + \normm{v}_{L^\infty}) \normm{u-v}_{L^p}
\end{align*}
Thus, we have a fixed point $z$ which is a solution of the equation \eqref{nseqz}.
Thus, $z$ satisfies 
\begin{align} \label{eqforz}
\pt_t z = \Delta z - \mathbb{P}((\c \cdot \nabla)z +(w \cdot \nabla)z+ (z \cdot \nabla)w  + (z \cdot \nabla)z) 
\end{align}
in the sense of distribution. If $z_0$ is sufficiently smooth($H^2$), right hand of \eqref{eqforz} is also smooth which means $z$ is a strong solution of the equation \eqref{nseqz}.
\end{proof}

\subsection{Estimate of fluid term}

In this section we will get the estimate of $h$ in \eqref{BoltzR}. $h$ can be split into the order of $\e$.

\begin{eqnarray}
h&=&-\e^{-3/2} L_{\mathbf{u}}(f_1)
\\
&& + \e^{-1/2} (-v_{\mathbf{u}} \cdot \nabla f_1 - L_{\mathbf{u}}(f_2) + \G(f_1,f_1) ) \\
&&  + \e^{1/2} (-\pt_t f_1 - v_{\mathbf{u}} \cdot \nabla f_2 - \mathbf{u} \cdot \nabla f_1 + \G(f_1, f_{2}) + \G(f_2, f_{1}) )\\
&&+\e^{3/2} (-\pt_t f_2 - \mathbf{u} \cdot \nabla f_2+  \G(f_2,f_2)) \\
&=& \e^{-3/2}h_{-3/2} + \e^{-1/2}h_{-1/2} + \e^{1/2}h_{1/2} +
\e^{3/2}h_{3/2}
\end{eqnarray}
The following things are the main proposition of this section.

\begin{proposition}
$h_{-3/2}= h_{-1/2} = 0$ and $h_{1/2} \perp \ker L_{\c}$    
\end{proposition}

\begin{proof}
$h_{-3/2}$ and $h_{-1/2}$ is 0 because of the definition of $f_1$ and $f_2$.
$h_0 \perp \ker L_{\c}$ is also clear.
Then, we only need to prove $h_{1/2} \perp \ker L_{\c}$. Since $\ker L_{\c} = \{1, v_\c, \norm{v_{\c}}^2 \}\sqrt{\mu_{\c}}$, $\P_{\c} h_{1/2}$ can be written as
\begin{eqnarray}
(-\pt_t u \cdot v_\c - \c \nabla (u \cdot v_\c)) \sqrt{\mu_\c} - \P_\c (v_\c \cdot \nabla f_2)
\end{eqnarray}

\begin{eqnarray*}
\inn{v_{\c,k} \mathscr{A}_{ij},\sqrt{\mu_\c}} = 0 \quad& \inn{v_{\c,k} \mathscr{A}_{ij}, v_{\c,k} \sqrt{\mu_\c}} = 0  &\quad \inn{v_{\c,k} \mathscr{A}_{ij},\norm{v_\c}^2\sqrt{\mu_\c}} = 0
\end{eqnarray*}
Define
\begin{eqnarray*}
\inn{\mathscr{A}_{ij}, v_{\c,i} v_{\c,j} \sqrt{\mu_\c}} = -\kappa.
\end{eqnarray*}
Then, 
\begin{eqnarray*}
\inn{v_{\c,k} \mathscr{A}_{ij}, v_{\c,l} \sqrt{\mu_\c}} =- (\delta_{ik}\delta_{jl}+\delta_{il}\delta_{jk})\kappa.
\end{eqnarray*}
So 
\begin{eqnarray*}
P_\c \left(v_\c \cdot \nabla_x \frac{1}{2} \sum_{i,j=1}^{3} \mathscr{A}_{ij} \left[\pt_{x_i}u_{j} + \pt_{x_j}u_{i} \right] \right)&=& -\frac{1}{2} \kappa \sum_{i,j=1}^{3} \left[(\pt_i \pt_i u_j + \pt_j \pt_i u_i)v_{\c,j} + (\pt_j \pt_i u_j + \pt_j \pt_j u_i)v_{\c,i}\right]\\
&=&-\kappa \sum_{j=1}^{3} \left[(\Delta u_j + \pt_j \nabla \cdot u)v_{\c,j} \right]\\
&=& -\kappa \sum_{j=1}^{3}\Delta u_j v_j \sqrt{\mu}
\end{eqnarray*}
Moreover, by \cite{Bardos1991},
\begin{eqnarray*}
&&\P_\c (L^{-1}\left[\G\left(f_1,f_1\right)  \right]) =\sqrt{\mu_\c}\sum_{j=1}^3v_{\c,j}\left( (u \cdot \nabla) u_j - \pt_{j} \frac{1}{3}\norm{u}^2 \right)
\end{eqnarray*}
Thus,
$\P_\c h_{1/2} =0$ because of the construction of the $u$.
\end{proof}

\begin{proposition} 
\label{smallh}
For $i = \{1/2,  3/2 \}$, $\normm{h_{i}}_{L^p} \lesssim C(\normm{u_0}_{L^1}, \normm{u_0}_{L^\infty}, \normm{\nabla u_0}_{L^1}, \normm{\nabla u_0}_{L^\infty}, \norm{\c})$ for $2 \le p\le \infty$
\end{proposition}

\begin{proof}
Recall the $w$ of the \eqref{nseqw}. According to the \cite{Galdi2011},  $\normm{w}_{L^{r}} \lesssim \norm{\c}$ for all $3 \le r \le \infty$ and $\normm{\nabla w}_{L^{r}} \lesssim \norm{\c} $ for all $2 \le r \le \infty$. According to the section \ref{nseqestimate} and \cite{Enomoto2005}, \cite{Kobayashi1998}, $\normm{z}_{L^{s}}, \normm{\nabla z}_{L^{s}} \lesssim C(\norm{\c}, \normm{z_0}_{L^r}, \normm{\nabla z_0}_{L^r})$ for $r > 1$. 
Since $u = z+w$ and by the definition of $f_1$ and $f_2$ (\eqref{deff1} and \eqref{deff2}) $\normm{h_{i}}_{L^p} \lesssim C(\normm{u_0}_{L^1}, \normm{u_0}_{L^\infty}, \norm{\c})$ for $2 \le p\le \infty$ holds.
\end{proof}

\subsection{Estimate of nonlinear term} \label{nonlinearterm}

In this section we will estimate the $\tilde{L}(f)$ and $\G(f,f)$. Proving the following proposition is the goal of this subsection.
\begin{proposition} \label{nonlinearestimate}
Let $g_i(t,x,v)$, $i=1,2$ is a smooth functions. Then,
\begin{align} \label{gg1g2}
\normm{\nu^{-1/2}\G(g_1,g_2)}_{L^2_{t,x,v}} \lesssim& \normm{\ip g_1}_{L^2_t L^2_\nu} \normm{\omega g_2}_{L^\infty_{t,x,v}} + \normm{S_1 g_1}_{L^2_t L^3_{x}} \normm{\P_\c g_2}_{L^\infty_t L^6_{x}}\\ \nonumber
&  + \normm{S_1 g_1}_{L^2_t L^3_{x}} \normm{\e^{-1} \ipc g_2}_{L^\infty_t L^2_\nu}^{1/3} \normm{\e^{1/2} \ipc g_2}_{L^\infty_t L^\infty_{x,v}}^{2/3}, \\
\label{gg2g1}
\normm{\nu^{-1/2} \G(g_2,g_1)}_{L^2_{t,x,v}} \lesssim& \normm{\ip g_1}_{L^2_t L^2_\nu} \normm{\omega g_2}_{L^\infty_{t,x,v}} + \normm{S_1 g_1}_{L^2_t L^3_{x}} \normm{\P_\c g_2}_{L^\infty_t L^6_{x}}\\ \nonumber
&  + \normm{S_1 g_1}_{L^2_t L^3_{x}} \normm{\e^{-1} \ipc g_2}_{L^\infty_t L^2_\nu}^{1/3} \normm{\e^{1/2} \ipc g_2}_{L^\infty_t L^\infty_{x,v}}^{2/3}, \\
\label{gg1g2inf}
\normm{\G(g_1,g_2)}_{L^\infty_{x,v}} \lesssim & \normm{\omega g_1}_{L^\infty_{x,v}} \normm{\omega g_2}_{L^\infty_{x,v}},
\end{align}
where $\omega(v) = exp(\b \norm{v}^2) $.
\end{proposition}
\begin{remark}
It is important that the right-hand side of \eqref{gg1g2} and \eqref{gg2g1} are not symmetric; the key point is that the term $\normm{g_1}_{L^\infty}$ does not appear in them.
\end{remark}

\begin{lemma} \label{lemma52}
Let $g_i(t,x,v)$, $i=1,2,3$ be smooth functions. Then,
\begin{eqnarray*}
&&\int_{\R^3}\G(g_1,g_2)g_3 \dd v  \\
&&\lesssim \left[\int_{\R^3} \nu g_1^2 \dd v \right]^{1/2}\left[\int_{\R^3} g_2^2 \dd v \right]^{1/2} \left[\int_{\R^3} \nu g_3^2 \dd v \right]^{1/2}\\
&& + \left[\int_{\R^3}  g_1^2 \dd v \right]^{1/2} \left[\int_{\R^3} \nu g_2^2 \dd v \right]^{1/2} \left[\int_{\R^3} \nu g_3^2 \dd v \right]^{1/2}.
\end{eqnarray*}
\end{lemma}
\begin{proof}
This is the Lemma 2.3 in  \cite{Guo2002}.
\end{proof}

\begin{corollary} \label{cor53}
Let $\omega_{\b}(v) = exp(\b \norm{v}^2) $ then,
$$
\normm{\nu^{-1/2} \G(g_1,g_2)}_{2} \lesssim \normm{g_1}_{\nu} \normm{\omega_{\b} g_2}_{\infty}
$$
$$
\normm{\nu^{-1/2} \G(g_1,g_2)}_{2} \lesssim  \normm{\omega_{\b} g_1}_{\infty} \normm{g_2}_{\nu}
$$
for all $\b>0$.
\end{corollary}
\begin{proof}
Let's put $g_3 = \nu^{-1} \G(g_1,g_2)$ to the lemma \ref{lemma52} then we can get,
\begin{eqnarray*}
&&\int_{\R^3}\nu^{-1} \G(g_1,g_2)^2 \dd v  \\
&&\lesssim \left[\int_{\R^3} \nu g_1^2 \dd v \right]^{1/2}\left[\int_{\R^3} g_2^2 \dd v \right]^{1/2} \left[\int_{\R^3}  \nu^{-1} \G(g_1,g_2)^2 \dd v \right]^{1/2}\\
&& + \left[\int_{\R^3}  g_1^2 \dd v \right]^{1/2} \left[\int_{\R^3} \nu g_2^2 \dd v \right]^{1/2} \left[\int_{\R^3} \nu^{-1} \G(g_1,g_2)^2 \dd v \right]^{1/2}.
\end{eqnarray*}
In addition,
\begin{eqnarray*}
\int_{\R^3} \nu f^2 \dd v &\le& \int_{\R^3} \omega_{\b}^{-2} \nu (\omega_{\b} f)^2 \dd v \\
&\le& \normm{\omega_{\b} f}_{\infty}^2 \int_{\R^3} \omega_{\b}^{-2} \nu \dd v \\
&\le&  C(\b) \normm{\omega_{\b} f}_{\infty}^2.
\end{eqnarray*}
So,
\begin{eqnarray*}
\normm{\nu^{-1/2} \G(g_1,g_2)}_{2}^{2} &=& \iint\nu^{-1} \G(g_1,g_2)^2 \dd v dx\\
&\le& \int\left[ \int_{\R^3} \nu g_1^2 \dd v \int_{\R^3} g_2^2 \dd v + \int_{\R^3}  g_1^2 \dd v \int_{\R^3} \nu g_2^2 \dd v \right] dx\\
&\lesssim& \min\left\{ \normm{g_1}_{\nu} \normm{\omega_{\b} g_2}_{\infty},  \normm{\omega_{\b} g_1}_{\infty} \normm{g_2}_{\nu} \right\}^{2}.
\end{eqnarray*}
\end{proof}

\begin{proof} (proof of Proposition \ref{nonlinearestimate})
First,
\begin{align*}
\normm{\G(g_1,g_2)}_{L^\infty_{x,v}} \lesssim & \normm{\omega g_1}_{L^\infty_{x,v}} \normm{\omega g_2}_{L^\infty_{x,v}}
\end{align*}
is straightforward by the definition of $\G$.
To get $L^2$ estimate, we can decompose $\G(g_1,g_2)$ and  $\G(g_2,g_1)$ into
\begin{align*}
\G(g_1,g_2) =& \G(\P_\c g_1, \P_\c g_2) + \G(\P_\c g_1, \ipc g_2) + \G(\ipc g_1, g_2), \\
\G(g_2,g_1) =& \G(\P_\c g_2, \P_\c g_1) + \G(\ipc g_2, \P_\c g_1) + \G(g_2, \ipc g_1)
.
\end{align*}
According to the corollary \ref{cor53}
\begin{align*}
\normm{\nu^{-1/2}\G(g_2, \ipc g_1)}_{L^2_{x,v}} &\lesssim \normm{\ipc g_1}_{\nu} \normm{\omega g_2}_{L^\infty_{x,v}}, \\
\normm{\nu^{-1/2} \G(\ipc g_1, g_2)}_{L^2_{x,v}} &\lesssim \normm{\ipc g_1}_{\nu} \normm{\omega g_2}_{L^\infty_{x,v}}.
\end{align*}
We need to estimate $\G(\P_\c g_1, \ipc g_2)$ and $\G(\ipc g_2, \P_\c g_1)$ without $\normm{g_1}_{\infty}$.
Since $\P_\c$ is a projection onto $\ker L_\c$, the following relation holds:
\begin{align*}
\int_{\R^3} \nu \P_\c g_1^2 \dd v \sim 
\int_{\R^3} \P_\c g_1^2 \dd v
\sim \P_\c g_1^2.
\end{align*}
By Lemma \ref{lemma52},
\begin{align*}
\int_{\R^3}\nu^{-1} {\G(\P_\c g_1, \ipc g_2)}^2 \dd v \lesssim& \int_{\R^3} \nu \P_\c g_1^2 \dd v \int_{\R^3} \ipc g_2^2 \dd v \\
&+ \int_{\R^3}  \P_\c g_1^2 \dd v \int_{\R^3} \nu \ipc g_2^2 \dd v \\
\lesssim &  \P_\c g_1^2  \int_{\R^3} \nu \ipc g_2^2 \dd v,\\
\int_{\R^3}\nu^{-1} {\G(\ipc g_2,\P_\c g_1)}^2 \dd v \lesssim&  \P_\c g_1^2 \int_{\R^3} \nu \ipc g_2^2 \dd v.
\end{align*}
By Proposition \ref{l3estimate},
\begin{align*}
\normm{\nu^{-1/2}\G(\P_\c g_1, \ipc g_2)}_{L^2_t, L^2_{x,v}}^2 =& \iint_{[0,T] \times \Omega} \int_{\R^3}\nu^{-1} {\G(\P_\c g_1, \ipc g_2)}^2 \dd v \dd x \dd t \\
\lesssim & \iint_{[0,T] \times \Omega} \left[ \P_\c g_1^2 \int_{\R^3} \nu \ipc g_2^2 \dd v \right] \dd x \dd t \\
\lesssim & \iint_{[0,T] \times \Omega} \left[ S_1 g_1^2 \int_{\R^3} \nu \ipc g_2^2 \dd v \right] \dd x \dd t \\
& + \iint_{[0,T] \times \Omega} \left[S_2 g_1^2 \int_{\R^3} \nu \ipc g_2^2 \dd v \right] \dd x \dd t.
\end{align*}
\begin{align*}
\iint_{[0,T] \times \Omega} \left[ S_2 g_1^2 \int_{\R^3} \nu \ipc g_2^2 \dd v \right] \dd x \dd t &\lesssim \iint_{[0,T] \times \Omega} \left[ S_2 g_1^2  \sup_{v}(\omega g_2)^2 \right] \dd x \dd t \\
&\lesssim \normm{S_2 g_1}_{L^2_t L^2_{x}}^2 \normm{ \omega g_2}_{L^\infty_{t}L^\infty_{x,v}}^2 \\
& \lesssim
\normm{\ipc g_1}_{L^2_t L^2_{x,v}}^2 \normm{ \omega g_2}_{L^\infty_{t} L^\infty_{x,v}}^2.
\end{align*}
By Holder inequality,
\begin{align*}
&\iint_{[0,T] \times \Omega} \left[ S_1 g_1^2 \int_{\R^3} \nu \ipc g_2^2 \dd v \right] \dd x \dd t \\
&\lesssim  \int_{[0,T]}  \left(\int_{\Omega} S_1 g_1^3 \dd x \right)^{2/3} \left( \int_{\Omega} \left(  \int_{\R^3} \nu \ipc g_2^2 \dd v \right)^3 \dd x \right)^{1/3} \dd t \\
& \lesssim  \int_{[0,T]}  \left(\int_{\Omega} S_1 g_1^3 \dd x \right)^{2/3} \left(\normm{\omega g_2}_{L^\infty_{x,v}}^4  \normm{\ipc g_2}_{\nu}^2\right)^{1/3} \dd t \\ 
&\lesssim \normm{S_1 g_1}^2_{L^2_{t}L^{3}_{x}} \normm{\ipc g_2}_{L^{\infty}_t L^2_{\nu}}^{2/3} \normm{\omega g_2}^{4/3}_{L^{\infty}_{t,x,v}}.
\end{align*}
Since $\P_\c g_1$ and $\P_\c g_2$ is in the $\ker L_\c$,
\begin{align*}
\normm{\nu^{-1/2}\G(\P_\c g_1, \P_\c g_2)}_{L^2_{t} L^2_{x,v}} \lesssim  &
\normm{\P_\c g_1 \P_\c g_2}_{L^2_{t} L^2_{x,v}} \\
\lesssim & \normm{\ipc g_1}_{L^2_t L^2_{\nu}} \normm{\omega g_2}_{L^\infty_{t,x,v}} + \normm{S_1 g_1}_{L^2_t L^3_{x}} \normm{\P_\c g_2}_{L^\infty_t L^6_{x}}
\end{align*}
\end{proof}

\begin{corollary} \label{cor22}
The following inequalities holds:
\begin{align*}
\normm{\tilde{L}(f)}_2 \lesssim& \normm{f_1 + \e f_2}_{\infty} \normm{f}_2, \\
\normm{\pt_t \tilde{L}(f)}_2 \lesssim & \normm{f_1 + \e f_2}_{\infty} \normm{\pt_t f}_2 +\normm{\pt_t f_1 + \e \pt_t f_2}_{\infty} \normm{f}_2, \\
\int_0^t \normm{\G(f,f)}_2^2
\lesssim& \normm{S_1 f}_{L^2_t, L^3_{x}}^2 \normm{\P_\c f}_{L^\infty_t, L^6_{x}}^2 +\e\normm{\e^{-1}\ipc f}_{L^2_t, L^2_{x,v}}^2 \normm{\e^{1/2} \omega f}_{L^\infty_{t}, L^\infty_{x,v}}^2 \\
&+ \normm{S_1 f}_{L^2_t, L^3_{x}}^2 \normm{\e^{-1} \ipc f}_{L^\infty_t, L^2_{x}}^{2/3} \normm{\e^{1/2} \omega f}_{L^\infty_t, L^\infty_{x}}^{4/3}, \\
\int_0^t \normm{\pt_t \G(f,f)}_2^2 \lesssim&  \normm{S_1 \pt_t f}_{L^2_t, L^3_{x}}^2 \normm{\P_\c f}_{L^\infty_t, L^6_{x}}^2 +\e\normm{\e^{-1}\ipc \pt_t f}_{L^2_t, L^2_{x,v}}^2 \normm{\e^{1/2} \omega f}_{L^\infty_{t}, L^\infty_{x,v}}^2 \\
&+ \normm{S_1 \pt_t f}_{L^2_t, L^3_{x}}^2 \normm{\e^{-1} \ipc f}_{L^\infty_t, L^2_{x}}^{2/3} \normm{\e^{1/2} \omega f}_{L^\infty_t, L^\infty_{x}}^{4/3}  
\end{align*}
\end{corollary}

\section{Proof of the main theorem}
The aim of this section is to prove the main result (Theorem \ref{mainthm}) of this paper.

\begin{definition}
Define dissipation and energy as
\begin{eqnarray*}
e(t)[f]&:=& \normm{f}^2_{L^2}(t) + \normm{\pt_t f}^2_{L^2}(t), \\
\mathscr{E}(T)[f] &:=& \sup_{0 \le t \le T} \normm{f}^2_{L^2}(t) + \sup_{0 \le t \le T} \normm{\pt_t f}^2_{L^2}(t), \\
\mathscr{D}(T)[f] &:=& \frac{1}{\e^2} \int_{0}^{T} \normm{\ip f}_{\nu}^2 \dd t + \frac{1}{\e^2} \int_{0}^{T} \normm{\pt_t \ip f}_{\nu}^2 \dd t \\
&& + \frac{1}{\e} \int_{0}^{T} \norm{(1-P^{\mathbf{u}}_{\g}) f}^2 \dd t + \frac{1}{\e} \int_{0}^{T} \norm{(1-P^{\mathbf{u}}_{\g}) \pt_t f}^2 \dd t.
\end{eqnarray*}
\end{definition}

\begin{proof}
Consider ${f}^l$ solving, for $l \in \mathbb{N}$,
\begin{align}
&\e \pt_t f^{l+1} + v \cdot \nabla_x f^{l+1} + \e^{-1}L(f^{l+1}) = h + \tilde{L}(f^{l}) + \e^{1/2} \G(f^{l},f^{l})\\
&f^{l+1}_{\g_-} = \mathscr{P}^{\c}_{\g} f^{l+1} + \e^{1/2} r, \quad f^{l+1} \big|_{t=0} =f_0.
\end{align}
Here we set : $f^0(t,x,v) := 0$
Take the time integration to Proposition \ref{energy} and use the Corollary \ref{cor22} then we can get:
\begin{align} \label{eq75}
&e[f^{l+1}](t) + \mathscr{D}[f^{l+1}](t) \\ \nonumber
& \lesssim 
\normm{f_0}^2_2 + \normm{\pt_t f_0}^2_2 + \normm{f_0}^2_\infty + \normm{\pt_t f_0}^2_\infty  + \e^{-2} \normm{\ipc f_0}^2_{\nu} + \e^{-1} \normm{(1-\mathscr{P}^{\c}_{\g}) f_0}^2_{2,+} \\ \nonumber
&+ (1+\e^{1/2})(\norm{ r}_{L^2_{t},2,-}^2 + \norm{\pt_t r}_{L^2_{t}, 2,-}^2)
+ (\e\norm{\c})^{-1}(\normm{z_\g( r)}_{L^2_{t}L^2_{x,v}}^2 + \normm{z_\g(\pt_t r)}_{L^2_{t}L^2_{x,v}}^2) \\ \nonumber
& + \normm{\P_\c h_{3/2}}_{L^2_{t}L^2_{x,v} } + \normm{\P_\c \pt_t h_{3/2}}_{L^2_{t} L^2_{x,v}}+ \normm{\ipc h}_{L^2_{t}L^2_{x,v}}^2  + \normm{\ipc \pt_t h}_{L^2_{t}L^2_{x,v}}^2 \\ \nonumber
& + (\normm{f_1 + \e f_2}_{L^\infty_{t,x,v}}^2 + \normm{\pt_t f_1 + \e \pt_t f_2}_{L^\infty_{t,x,v}}^2) \int_0^t e[f^{l}](t)\dd t \\ \nonumber
& +\e^{} (\normm{S_1 f^{l}}_{L^2_t L^3_{x,v}}^2+\normm{\pt_t S_1 f^{l}}_{L^2_t L^3_{x,v}}^2) \normm{\P_{\c} f^{l}}_{L^\infty_t L^6_{x,v}}^2 \\ \nonumber
& +\e^{2} \normm{\e^{-1} \pt_t \ipc \e^{1/2} f^{l}}_{L^2_{t,x,v}}^2 \normm{f^{l}}_{L^\infty_{t,x,v}}^2 \\ \nonumber
&+ \e^{} (\normm{S_1 f^{l}}_{L^2_t L^3_{x,v}}^2 + \normm{\pt_t S_1 f^{l}}_{L^2_t L^3_{x,v}}^2) \normm{\e^{-1} \ipc f^{l}}_{L^\infty_t L^2_{x,v}}^{2/3} \normm{\e^{1/2} f^{l}}_{L^\infty_{t,x,v}}^{4/3}.
\end{align}
By \eqref{eq27} and Proposition \ref{smallh} the following term is bounded by $C_{\text{initial}}$,
\begin{eqnarray*}
&&
(1+\e^{\frac 1 2})(\norm{r}_{2,-}^2 + \norm{\pt_t r}_{2,-}^2) + (\e\norm{\c})^{-1}(\normm{z_\g(r)}_2^2 + \normm{z_\g(\pt_t r)}_2^2) \\
&& + \e^{-2}\normm{P_\c h}_2^2 + \e^{-2} \normm{P_\c \pt_t h}_2^2 + \normm{h}_2^2+ \normm{\pt_t h}_2^2 \\
&& + \normm{f_1+ \e f_2}^2_{\infty} +\normm{\pt_t f_1+ \e \pt_t f_2}^2_{\infty} +  \norm{\c},
\end{eqnarray*}
if we set $C_{\text{initial}}$ as 
\begin{eqnarray*}
C_{\text{initial}} := \max \{ \normm{u_0}, \normm{\nabla u_0}, \normm{f_0}, \normm{\pt_t u_0}, \normm{\pt_t \nabla u_0}, \normm{\pt_t f_0}, \norm{\c}\},
\end{eqnarray*}
where $\normm{\cdot}$ implies both $L^\infty$ and $L^1$ norm and $\pt_t$ came from the equation.

Integrating \eqref{eq75} by time from $0$ to $T$, then we can get,
\begin{align*} 
&e[f^{l+1}](T) + \mathscr{D}[f^{l+1}](T)  \\
&\le
\int_{0}^T C_{\text{initial}} \dd s + C_{\text{initial}}  \int_{0}^T e[f^{l}](s) \dd s \\
& +c_1\e^{} (\normm{S_1 f^{l}}_{L^2_t L^3_{x,v}}^2+\normm{\pt_t S_1 f^{l}}_{L^2_t L^3_{x,v}}^2) \normm{\P_{\c} f^{l}}_{L^\infty_t L^6_{x,v}}^2 \\
& +c_1\e^{2} \normm{\e^{-1} \pt_t \ipc \e^{1/2} f^{l}}_{L^2_{t,x,v}}^2 \normm{f^{l}}_{L^\infty_{t,x,v}}^2 \\
&+ c_1\e^{} (\normm{S_1 f^{l}}_{L^2_t L^3_{x,v}}^2 + \normm{\pt_t S_1 f^{l}}_{L^2_t L^3_{x,v}}^2) \normm{\e^{-1} \ipc f^{l}}_{L^\infty_t L^2_{x,v}}^{2/3} \normm{\e^{1/2} f^{l}}_{L^\infty_{t,x,v}}^{4/3},
\end{align*}
where $c_1$ is a fixed constant. 

By the Proposition \ref{linfestimate},
\begin{align*}
\normm{\e^{1/2} \omega f^{l+1}(t)}_{\infty}^2 \le& c_2 \normm{\e^{1/2} \omega f_0}_{\infty}^2 + c_2\sup_{0\le s\le t } \normm{\e^{1/2} \omega r(s)}_{\infty}^2 + c_2\e^{3} \sup_{0\le s\le t } \normm{ \inn{v}^{-1} \omega g(s)}_{\infty}^2 \\
&+ c_2\normm{\P_\c f^l}_{L^\infty_t L^{6}_{x}}^2 + c_2\e^{-2} \normm{ \ipc f}_{L^\infty_t \nu}^2 \\
\le& \int_{0}^T C_{\text{initial}} \dd s + \mathscr{D}[f^{l}](t)+ c_2\normm{\P_\c f^l}_{L^\infty_t L^{6}_{x}}^2 + C_{\text{initial}}  \int_{0}^{t}e(s)[f^{l}] \dd s  \\
& +c_2\e^{} (\normm{S_1 f^{l}}_{L^2_t L^3_{x,v}}^2+\normm{\pt_t S_1 f^{l}}_{L^2_t L^3_{x,v}}^2) \normm{\P_{\c} f^{l}}_{L^\infty_t L^6_{x,v}}^2 \\
& +c_2\e^{2} \normm{\e^{-1} \pt_t \ipc \e^{1/2} f^{l}}_{L^2_{t,x,v}}^2 \normm{f^{l}}_{L^\infty_{t,x,v}}^2 \\
&+ c_2\e^{} (\normm{S_1 f^{l}}_{L^2_t L^3_{x,v}}^2 + \normm{\pt_t S_1 f^{l}}_{L^2_t L^3_{x,v}}^2) \normm{\e^{-1} \ipc f^{l}}_{L^\infty_t L^2_{x,v}}^{2/3} \normm{\e^{1/2} f^{l}}_{L^\infty_{t,x,v}}^{4/3}.
\end{align*}
where $c_2$ is a fixed constant. 

By Proposition \ref{l6estimate} the $L^6$ bound is
\begin{align*}
\normm{\P_\c f^{l+1}}_{L^\infty_t L^{6}_{x}}^2 \le& c_3 \e^{-2}\normm{\ipc f^{l}}_{L^\infty_t L^2_{\nu}}^2+
c_3 \normm{ g{{\nu}^{-\frac 1 2}}}_{L^\infty_t L^2_{x,v}}^2 + c_3 \e^2 \normm{\pt_t f^{l+1}}_{L^\infty_t L^2_{x,v}}^2 \\
&+ c_3 \e^{-1}\norm{(1-\mathscr{P}^{\c}_{\g})f^{l}}_{L^\infty_t  2,+}^2 + c_3 \e \norm{r}_\infty^2+ o(1)[\e^{1/2}\normm{f^{l}}_{\infty}]^2 \\
\le& \int_{0}^T C_{\text{initial}} \dd s + \mathscr{D}[f^{l}](t) + \e c_3 E[f^{l}](T) + o(1) \normm{\e^{1/2} f^{l}}_{L^\infty_{t,x,v}}\\
& + C_{\text{initial}}  \int_{0}^{t}e(s)[f^{l}] \dd s  \\
& + c_3 \e^{} (\normm{S_1 f^{l}}_{L^2_t L^3_{x,v}}^2+\normm{\pt_t S_1 f^{l}}_{L^2_t L^3_{x,v}}^2) \normm{\P_{\c} f^{l}}_{L^\infty_t L^6_{x,v}}^2 \\
& + c_3 \e^{2} \normm{\e^{-1} \pt_t \ipc \e^{1/2} f^{l}}_{L^2_{t,x,v}}^2 \normm{f^{l}}_{L^\infty_{t,x,v}}^2 \\
& + c_3 \e^{} (\normm{S_1 f^{l}}_{L^2_t L^3_{x,v}}^2 + \normm{\pt_t S_1 f^{l}}_{L^2_t L^3_{x,v}}^2) \normm{\e^{-1} \ipc f^{l}}_{L^\infty_t L^2_{x,v}}^{2/3} \normm{\e^{1/2} f^{l}}_{L^\infty_{t,x,v}}^{4/3},
\end{align*}
where $c_3$ is a fixed constant. 

By Proposition \ref{l3estimate}, the $L^3$ bound is
\begin{align*}
\normm{S_1 f}_{L^2_t L^3_x}^2 \le& c_4 \normm{f}_{L^2_t L^2_x}^2 + c_4 \normm{g}_{L^2_t L^2_x}^2 + c_4 \normm{f_0 }_{L^2_t L^2_x}^2 + c_4 \normm{ v \cdot \nabla f_0}_{L^2_t L^2_x}^2 + c_4 \norm{f_0}_{\g}^2 + c_4 \norm{f}_{L^2_t \g}^2 \\
\le& \int_{0}^T C_{\text{initial}} \dd s + c_4 \mathscr{D}[f^{l}](t) +C_{\text{initial}}  \int_{0}^{t}e(s)[f^{l}] \dd s\\
& + c_4 \e^{} (\normm{S_1 f^{l}}_{L^2_t L^3_{x,v}}^2+\normm{\pt_t S_1 f^{l}}_{L^2_t L^3_{x,v}}^2) \normm{\P_{\c} f^{l}}_{L^\infty_t L^6_{x,v}}^2 \\
& + c_4 \e^{2} \normm{\e^{-1} \pt_t \ipc \e^{1/2} f^{l}}_{L^2_{t,x,v}}^2 \normm{f^{l}}_{L^\infty_{t,x,v}}^2 \\
&+ c_4 \e^{} (\normm{S_1 f^{l}}_{L^2_t L^3_{x,v}}^2 + \normm{\pt_t S_1 f^{l}}_{L^2_t L^3_{x,v}}^2) \normm{\e^{-1} \ipc f^{l}}_{L^\infty_t L^2_{x,v}}^{2/3} \normm{\e^{1/2} f^{l}}_{L^\infty_{t,x,v}}^{4/3},
\end{align*}
where $c_4$ is a fixed constant.

Set $M= 2 T_0 C_{\text{initial}}$ and define space as
\begin{align*}
H := \Bigg\{f :& e(t) +\mathscr{D}(t) \le M \exp{(Mt)}, \normm{\P_\c f^{}}_{L^\infty_t L^{6}_{x}}^2 \le 2(c_3+1) M\exp{(MT_0)}, \\
& \normm{\e^{1/2} \omega f^{}(t)}_{\infty}^2 \le 2(c_2c_3 + c_2 +1) M\exp{(MT_0)}, \normm{S_1 f}_{L^2_t L^3_x}^2 \le 2(c_4+1) M\exp{(MT_0)}
  \Bigg\}
\end{align*}

Let's set $\e$ small enough to satisfy
$\e^{} \max\{c_1,c_2,c_3,c_4,1\} ^2 M\exp{(MT_0)} \ll 1$

Then according to the above relation, the map $\Phi : f^l \to f^{l+1}$ is bounded map from $H$ to $H$ by Gronwall's lemma. In addition,
\begin{eqnarray*}
\G(g_1,g_1)-\G(g_2,g_2) &=& \G(g_1-g_2,g_1) +\G(g_2,g_1-g_2) \\
\pt_t \G(g_1,g_1)-\pt_t \G(g_2,g_2) &=& \G(\pt_t g_1 - \pt_t g_2,g_1) + \G(\pt_t g_2,g_1-g_2) + \G( g_1 -g_2,\pt_t g_1 ) + \G(g_2,\pt_t g_1 -\pt_t g_2)
\end{eqnarray*}
Thus,
$\Phi$ is a contraction map.
Thus $\Phi$ has a unique fixed point which is a solution of the equation.
\end{proof}


\begin{thebibliography}{20}


\bibitem{Asano1983}
Asano, K., Ukai, S.: The Euler limit and the initial layer of the nonlinear Boltzmann equation. Hokkaido Math. J. 12, 303–324 (1983)


\bibitem{Bae2010}
Bae, H.-O., Roh, J.: Stability for the 3D Navier–Stokes equations with nonzero far field velocity on exterior domains. J. Math. Fluid Mech. 14, 117-139 (2010)


\bibitem{Bardos1991}
Bardos, C., Ukai, S.: The classical incompressible Navier–Stokes limit of the Boltzmann equation. Math. Mod. Meth. Appl. Sci. 1, 235 (1991)


\bibitem{Bardos1991-1}
Bardos, C., Golse, F., Levermore, D.: Fluid Dynamic Limits of Kinetic Equations I: Formal Derivations. J. Statistical Physics  63, 323–344 (1991)

\bibitem{Bardos1993}
Bardos, C., Golse, F., Levermore, D.: Fluid dynamical limits of kinetic equations, II: convergence proofs for the Boltzmann equation. Comm. Pure Appl. Math. 46, 667–753 (1993)



\bibitem{Borchers1991}
Borchers, W., Miyakawa, T.: Algebraic $L^2$ decay for Navier–Stokes flows in exterior domains, II.
Hiroshima Math. J. 21, 621–640 (1991)


\bibitem{Caflisch1980}
Caflisch, R.E.: The fluid dynamical limit of the nonlinear Boltzmann equation. Comm. Pure Appl. Math. 33, 651–666 (1980)

\bibitem{Cao2023}
Cao, Y., Jang, J., Kim, C.: Passage from the Boltzmann equation with diffuse boundary to the incompressible Euler equation with heat convection. J. Differential Equations 366, 565–644. (2023)

\bibitem{Cercignani1994}
Cercignani, C., Illner, R., Pulvirenti, M.: The Mathematical Theory of Dilute Gases. Springer, New
York (1994)

\bibitem{Masi1989}
De Masi, A., Esposito, R., Lebowitz, J.L.: Incompressible Navier–Stokes and Euler Limits of the Boltzmann Equation. Comm. Pure Appl. Math. 42, 1189–1214 (1989)

\bibitem{Diperna1989}
DiPerna, R.J., Lions, P.L.: On the Cauchy problem for Boltzmann equations: global existence and weak stability. Ann. Math. 130, 321–366 (1989)

\bibitem{Duan2008}
Duan, J.: On the Cauchy problem for the Boltzmann equation in the whole space: Global existence and uniform stability in $L^2_{\xi} (H^N_x)$. J. Differential Equations 244, 3204-3234 (2008)

\bibitem{Duan2012}
Duan, J., Yang, T., Zhao, H.
: The Vlasov-Poisson-Boltzmann system in the whole space: the hard potential case.
J. Differential Equations 252, 6356–6386
 (2012)


\bibitem{Enomoto2005}
Enomoto, Y., Shibata, Y.: On the rate of decay of the Oseen semigroup in exterior domains and its application to Navier–Stokes equation. J. Math. Fluid Mech. 7, 339–367 (2005)


\bibitem{Esposito2013}
Esposito, R., Guo, Y., Kim, C., Marra, R.: Non-isothermal boundary in the Boltzmann theory and Fourier law. Comm. Math. Phys. 323, 177–239 (2013)


\bibitem{Esposito2018-apde}
Esposito, R., Guo, Y., Kim, C., Marra, R. : Stationary Solutions to the Boltzmann Equation in the Hydrodynamic Limit.
Ann. PDE (2018)

\bibitem{Esposito2018-cmp}
Esposito, R., Guo, Y., Marra, R. :
Hydrodynamic Limit of a Kinetic Gas Flow Past an Obstacle. Comm. Math. Phys. 364, 765–823 (2018)


\bibitem{Esposito2020}
Esposito, R., Guo, Y., Kim, C., Marra, R.: Diffusive limits of the Boltzmann equation in bounded domain. Ann. Appl. Math. 36, no. 2, 111–185. (2020)

\bibitem{Evans1998}
Evans, L.C.: Partial Differential Equations, Graduate Studies in Mathematics, vol. 19. American Mathematical Society, Providence (1998)


\bibitem{Galdi2011}
Galdi, G.P.:
An Introduction to the Mathematical Theory of the Navier-Stokes Equations Steady-State Problems. Springer (2011)

\bibitem{Giga1986}
Giga, Y.: Solutions for semilinear parabolic equations in $L^p$ and regularity of weak solutions of the Navier-Stokes system. J. Differential Equations 62, 186-212 (2008)




\bibitem{Glassey1996}
Glassey, R.: The Cauchy Problem in Kinetic Theory. SIAM (1996)


\bibitem{Golse2004}
Golse, F., Saint-Raymond, L.:The Navier–Stokes limit of theBoltzmann equation for bounded collision kernels. Invent. Math. 155, 81–161 (2004)

\bibitem{Guo2002}
Guo, Y.: The Vlasov-Poisson-Boltzmann system near Maxwellians, Comm. Pure Appl. Math. 55
, 1104–1135 (2002)

\bibitem{Guo2006}
Guo, Y. : Boltzmann Diffusive Limit Beyond the Navier-Stokes Approximation. Comm. Pure Appl. Math, 59, 0626–0687 (2006)


\bibitem{Guo2010-arma}
Guo, Y.: Decay and continuity of the Boltzmann equation in bounded domains. Arch. Ration. Mech. Anal. 197, 713–809 (2010)

\bibitem{Guo2010}
Guo, Y., Jang, J.: Global Hilbert expansion for the Vlasov–Poisson–Boltzmann system. Comm. Math. Phys. 299, 469–501 (2010)



\bibitem{Guo2010_2}
Guo, Y., Jang, J., Jiang, N.: Acoustic limit for the Boltzmann equation in optimal scaling. Comm. Pure Appl. Math. 63, 337–361 (2010)


\bibitem{Hilbert1916}
Hilbert, D.: Begründung der kinetischen Gastheorie. Math. Ann. 72, 331–407 (1916/17)

\bibitem{Hilbert_}
Hilbert, D.: Grundzügeiner allgemeinen Theorie der linearen Integralgleichungen. Chelsea, New York

\bibitem{Hilbert1900}
Hilbert, D.: Mathematical Problems. Göttinges Nachrichten, pp. 253–297 (1900)


\bibitem{Huang2010}
Huang, F., Wang, Y., Yang, T.: Hydrodynamic limit of the Boltzmann equation with contact discontinuities. Comm. Math. Phys. 295, 293–326 (2010)

\bibitem{Jang2009}
Jang, J.: Vlasov-Maxwell-Boltzmann diffusive limit. Arch. Ration. Mech. Anal. 194, no. 2, 531–584. (2009)

\bibitem{Jang2021}
Jang, J., Kim, C.: Incompressible Euler limit from Boltzmann equation with diffuse boundary condition for analytic data. Ann. PDE 7, no. 2, Paper No. 22, 103 pp. (2021)


\bibitem{kato1984}
Kato, T.: Strong $L^p$-Solutions of the Navier-Stokes equation in $R^m$, with applications to weak solutions. Mathematische Zeitschrift. Vol 187, 471-480 (1984)

\bibitem{Kobayashi1998}
Kobayashi, T., Shibata, Y.On the Oseen equation in the three dimensional exterior domains. Math. Ann. 310, 1-45 (1998)


\bibitem{Lachowicz1987}
Lachowicz, M.: On the initial layer and the existence theorem for the nonlinear Boltzmann equation. Math. Methods Appl. Sci. 9(3), 27–70 (1987)


\bibitem{Lions2001}
Lions, P.-L., Masmoudi, N.: From the Boltzmann equations to the equations of incompressible fluid mechanics. I, II. Arch. Ration. Mech. Anal. 158(3), 173–193, 195–211 (2001)


\bibitem{Maslova1993}
Maslova, N.: Nonlinear Evolution Equations. World Scientific, Singapore (1993)


\bibitem{Masmoudi2003}
Masmoudi, N., Saint-Raymond, L.: From the Boltzmann equation to the Stokes–Fourier system in a bounded domain. Comm. Pure Appl. Math. 56(9), 1263–1293 (2003)



\bibitem{Nishida1978}
Nishida, T.: Fluid dynamical limit of the nonlinear Boltzmann equation to the level of the compressible Euler equation. Comm. Math. Phys. 61, 119–148 (1978)


\bibitem{Saint-Raymond2009}
Saint-Raymond, L.: Hydrodynamic limits of the Boltzmann equation. Lecture Notes in Mathematics, no. 1971. Springer, Berlin (2009)





\bibitem{Sone1971}
Sone, Y.: Asymptotic Theory of Flow of a Rarefied Gas over a Smooth Boundary II. In: Rarefied Gas Dynamics. Vol. II, pp. 737–749, ed. by D. Dini. Pisa: Editrice Tecnico Scientifica (1971)


\bibitem{Sone1987}
Sone, Y., Aoki K.: Steady Gas Flows Past Bodies at Small Knudsen Numbers - Boltzmann and Hydrodynamics Systems, Trans. Theory Stat. Phys., 16, 189-199 (1987)


\bibitem{Strain2006}
Strain, R. M.
: The Vlasov-Maxwell-Boltzmann system in the whole space.
Comm. Math. Phys. 268, 543–567 (2006)


\bibitem{Tolksdorf2020}
Tolksdorf, P., Watanabe, K.: The Navier–Stokes equations in exterior Lipschitz domains: $L^p$-theory. J. Differential Equations 269, 5765–5801 (2020)


\bibitem{Ukai1983}
Ukai, S., Asano, K.: Steady solutions of the Boltzmann equation for a gas flow past an obstacle I Existence. Arch. Ration. Mech. Anal. 84, 249–291 (1983)


\bibitem{Ukai1986}
Ukai, S., Asano, K.: Steady solutions of the Boltzmann equation for a gas flow past an obstacle. II. Stability. Publ. Res. Inst. Math. Sci. 22, 1035–1062 (1986)


\bibitem{Ukai2009}
Ukai, S., Yang, T., Zhao, H.: Stationary solutions to the exterior problems for the Boltzmann equation. I. Existence. Discrete Contin. Dyn. Syst. 23, 495–520 (2009)


\bibitem{Varnhorn2022}
Varnhorn, W. On the Poisson equation in exterior domains. Constructive Mathematical Analysis 5, No. 3, pp. 134-140 (2022)


\bibitem{Wigner2000}
Wiegner, M.: Decay estimates for strong solutions of the Navier–Stokes equations in exterior domains. Ann. Univ. Ferrara Sez. VII.
Sc. Mat. 46, 61–79 (2000)


\bibitem{Xia2017}
Xia J.: Exterior problem for the Boltzmann equation with temperature difference: asymptotic stability of steady solutions. J. Differential Equations. 262, 3642–3688 (2017)

\bibitem{Yu2005}
Yu, S.-H.: Hydrodynamic limits with shock waves of the Boltzmann equation. Comm. Pure Appl. Math. 58, 409–443 (2005)



\end{thebibliography}
\end{document}